\documentclass[12pt]{article}
\usepackage{fullpage}
\usepackage{amssymb,amsmath,amsfonts,amsthm,latexsym,enumerate,url,cases,mathrsfs,times,mathabx}
\numberwithin{equation}{section}
\usepackage{hyperref}
\hypersetup{colorlinks=true,citecolor=blue,linkcolor=blue,urlcolor=blue}

\newtheorem{tet}{Theorem}

\newtheorem{lem}{Lemma}

\newtheorem{kov}[tet]{Corollary}
\theoremstyle{remark}

\newtheorem{prob}{Problem}
\theoremstyle{definition}

\usepackage[left=50pt,right=50pt,top=30pt,bottom=80pt]{geometry}

\begin{document}

\title{Multiplicative complements II.}
\author{Anett Kocsis \thanks{E\" otv\" os Lor\' and University, Budapest, Hungary. Email: sakkboszi@gmail.com. Supported by the ÚNKP-21-1 New National Excellence Program of the Ministry for Innovation and Technology from the source of the National Research, Development and Innovation Fund. }, D\' avid Matolcsi \thanks{E\" otv\" os Lor\' and University, Budapest, Hungary. Email: matolcsidavid@gmail.com. Supported by the ÚNKP-21-1 New National Excellence Program of the Ministry for Innovation and Technology from the source of the National Research, Development and Innovation Fund.}, Csaba S\' andor \thanks{Department of Stochastics, Institute of Mathematics, Budapest University of
Technology and Economics, M\H{u}egyetem rkp. 3., H-1111, Budapest, Hungary. Department of Computer Science and Information Theory, Budapest University of Technology and Economics, M\H{u}egyetem rkp. 3., H-1111 Budapest, Hungary, MTA-BME Lend\"ulet Arithmetic Combinatorics Research Group,
  ELKH, M\H{u}egyetem rkp. 3., H-1111 Budapest, Hungary . Email: csandor@math.bme.hu.
This author was supported by the NKFIH Grants No. K129335.} and Gy\"orgy T\H ot\H os \thanks{ Faculty of Mathematics and Computer Science, Babe\c{s}-Bolyai University. }}

    \maketitle

\begin{abstract}
    In this paper we prove that if $A$ and $B$ are infinite subsets of positive integers such that every positive integer $n$ can be written as $n=ab$, $a\in A$, $b\in B$, then $\displaystyle \lim_{x\to \infty}\frac{A(x)B(x)}{x}=\infty $. We also prove many other results about sets like this.
\end{abstract}



    \section{Introduction}

The set of nonnegative integers is denoted by $\mathbb{N}$. The counting function of a set $A\subseteq \mathbb{N}$ is defined as $A(x)=|A\cap \{ 0,1,\dots ,x\}|$ for every $x\in \mathbb{N}$. Let $A,B\subseteq \mathbb{N}$. The sets $A$ and $B$ are said to be additive complements if every nonnegative integers $n$ can be written as $n=a+b$, $a\in A$, $b\in B$. Clearly, if $A,B\subseteq \mathbb{N}$ are additive complements, then $A(x)B(x)\ge x+1$ for every $x\in \mathbb{N}$, therefore $\displaystyle \liminf_{x\to \infty} \frac{A(x)B(x)}{x}\ge 1$.  In 1964, answering a question of Hanani, Danzer \cite{D} proved that this bound is sharp.

\begin{tet}[Danzer, 1964]
There exist infinite additive complements $A,B\subseteq \mathbb{N}$ such that
$$
\lim_{x\to \infty}\frac{A(x)B(x)}{x}=1.
$$
\end{tet}
In \cite{KMST} we introduced the concept of multiplicative complements. Let us denote by $\mathbb{Z}^+$ the set of positive integers and let $A_i \subseteq \mathbb{Z}^+$ for every $1\le i\le h$. The $h$-tuple $(A_1,\dots ,A_h)$ form multiplicative complements of order $h$ if every positive integers $n$ can be written as $n=a_1\dots a_h$, $a_i\in A_i$. For brevity, we will use the notation $MC_h$ for the set of multiplicative complements of order $h$. Similar to the additive complements, if $(A,B)\in MC_2$, then we have $A(x)B(x)\ge x$, therefore $$\liminf_{x\to \infty} \frac{A(x)B(x)}{x}\ge 1.$$ 
We show that, in contrast to the additive complements,
$$
\lim_{x\to \infty }\frac{A(x)B(x)}{x}=\infty
$$
for every infinite $(A,B)\in MC_2$.

We proved in \cite{KMST} the following statement:
\begin{tet}
For every $\varepsilon >0$, there exists infinite $(A,B)\in MC_2$ such that
$$
\liminf_{x\to \infty}\frac{ \max \{ A(x)B(x)\} }{\frac{x}{\log x}}\le 0.5+\varepsilon.
$$
\end{tet}
Clearly, for every $(A,B)\in MC_2$ we have $\log min\{ A(x),B(x) \}\le \log x$. It follows that
\begin{kov}
For every $\varepsilon >0$, there exists infinite $(A,B)\in MC_2$ such that
$$
\liminf_{x\to \infty}\frac{ \max \{ A(x)B(x)\} \log \min \{ A(x),B(x) \} }{x}\le 0.5+\varepsilon.
$$
\end{kov}
We show that this bound can not be improved.
\begin{tet}\label{liminf:maxlogmin}
For every infinite $(A,B)\in MC_2$,
$$
\liminf_{x\to \infty}\frac{ \max \{ A(x)B(x)\} \log \min \{ A(x),B(x) \} }{x}>0.5.
$$
\end{tet}
If $|A|=\infty $ and $|B|=\infty $, then $\displaystyle \lim _{x\to \infty}\frac{\min \{ A(x),B(x) \} }{\log \min \{ A(x),B(x)\} }=\infty $. It follows that
\begin{kov}\label{multD}
For every infinite $(A,B)\in MC_2$,
$$
\lim _{x\to \infty} \frac{A(x)B(x)}{x}=\infty .
$$
\end{kov}
In Theorem \ref{liminf:maxlogmin} we show that for infinite $(A,B)\in MC_2$, the fraction $\frac{ \max \{ A(x)B(x)\} \log \min \{ A(x),B(x) \} }{x}>0.5$ if $x$ is large enough. Next theorem shows that this fraction can be arbitrary large.
\begin{tet}\label{limsup:maxlogmin}
For every infinite $(A,B)\in MC_2$,
$$
\limsup_{x\to \infty}\frac{ \max \{ A(x)B(x)\} \log \min \{ A(x),B(x) \} }{x}=\infty.
$$
\end{tet}
We proved in \cite{KMST} that
\begin{tet}\label{limsup:max}
\begin{enumerate}
\item For every $(A,B)\in MC_2$,
$$
\limsup_{x\to \infty }\frac{max\{ A(x),B(x) \}}{\frac{x}{\sqrt{\log x}}}\ge \frac{1}{\sqrt{\pi }}.
$$
\item There exists an $(A,B)\in MC_2$ such that
$$
A(x)=\left( \frac{1}{\sqrt{\pi }}+o(1) \right) \frac{x}{\sqrt{\log x}}\quad \hbox{ and }\quad B(x)=\left( \frac{1}{\sqrt{\pi }}+o(1) \right) \frac{x}{\sqrt{\log x}}.
$$
\end{enumerate}
\end{tet}
In the next three theorems we consider the function $\min \{ A(x),B(x)\}$. It is easy to see that for every $\varepsilon >0$, there exists an $(A,B)\in MC_2$ such that $|B|<\infty $ and $\displaystyle \limsup_{x\to \infty} \frac{A(x)}{x}<\varepsilon $. On the other hand, if $|B|<\infty $ and $(A,B)\in MC_2$, then $\displaystyle \liminf_{x\to \infty} \frac{A(x)}{x}>0$. Therefore, the natural requirement is that $A(x)=o(x)$ and $B(x)=o(x)$.
\begin{tet}\label{liminffx}
Let $f(x)$ be a function such that $f(x)\to \infty $ as $x\to \infty $. Then there exists infinite $(A,B)\in MC_2$ such that $A(x)=o(x)$, $B(x)=o(x)$ and
$$
\liminf_{x\to \infty }\frac{\min\{ A(x),B(x) \}}{f(x)}=0.
$$
 \end{tet}
As a corollary, we get that in Corollary \ref{multD} the function $x$ can not be replaced by any other function $g(x)$, where $\frac{g(x)}{x}\to \infty $ as $x\to \infty$.
\begin{kov}
  Let $g(x)$ be a function such that $\displaystyle \lim_{x\to \infty}\frac{g(x)}{x}=\infty $. Then there exist an $(A,B)\in MC_2$ such that
  $$
  \limsup_{x\to \infty }\frac{A(x)B(x)}{g(x)}=0.
  $$
\end{kov}
The following theorem shows that for $(A,B)\in MC_2$ the function $\min\{ A(x),B(x) \}$ can be arbitrary large.
\begin{tet}\label{limsupfxno}
Let us suppose that for some function $f(x)>0$, $x\ge 1$, the series $\displaystyle \sum_{n=1}^{\infty}\frac{f(n)}{n^2}$ converges. Then there is no $(A,B)\in MC_2$ such that $A(x)=O(f(x))$ and $B(x)=o(x)$.
\end{tet}
As a corollary, we get
\begin{kov}
  Let $\varepsilon >0$. Then there is no infinite $(A,B)\in MC_2$, $B(x)=o(x)$ such that $A(x)=O\left( \frac{x}{(\log x)(\log \log x )^{1+\varepsilon }}\right)$.
\end{kov}
If the function $f(x)$ satisfies some smoothness conditions and $\displaystyle \sum_{n=1}^{\infty}\frac{f(n)}{n^2}=\infty $, then we can find $(A,B)\in \mathbb{Z}^+$ such that $A(x)=O(f(x))$ and $B(x)=o(x)$ as $x\to \infty$.
\begin{tet}\label{limsupfxyes}
Let us suppose that the function $f(x)>0$ satisfies the following conditions
\begin{enumerate}
  \item $f(x)$ is monotonically increasing for $x\ge 1$,
  \item $f(x)=o(\frac{x}{\log x})$ as $x\to \infty$,
  \item there exist $c_1,c_2>0$ such that $c_1\le \frac{\frac{f(x_2)}{f(x_1)}}{\frac{x_2}{x_1}}\le c_2$ for $1\le x_1\le x_2\le x_1^2$,
  \item $\sum_{n=1}^{\infty}\frac{f(n)}{n^2}=\infty $.
\end{enumerate}
Then there exists $(A,B)\in MC_2$ such that $A(x)=O(f(x))$ and $B(x)=o(x)$.
\end{tet}
As a corollary, we get
\begin{kov}
  There exists $(A,B)\in MC_2$ such that $A(x)=O(\frac{x}{(\log x)(\log \log x)})$ and $B(x)=o(x)$.
\end{kov}

In the last theorem, we consider $(A,B)\in MC_2$ such that $|B|< \infty$.
\begin{tet}\label{vegeseset}
\begin{enumerate}
\item There exists a $c_1>0$ such that for every $(A,B)\in MC_2$, $2\le |B|<\infty $, we have
$$
\frac{A(x)\log |B|}{x}\ge c_1
$$
for every $x\in \mathbb{Z}^+$.
\item There exists a constant $c_2$ such that for every $N\ge 3$, there exists $(A,B)\in MC_2$, $|B|=N$ such that
$$
\frac{A(x)\log \log |B|}{x}\le c_2
$$
for every $x\ge 10$.
\end{enumerate}
\end{tet}
We finally pose some open problems for further research. As a corollary of Theorem \ref{limsup:max}, we get the following statement:
\begin{kov}
There exists $(A,B)\in MC_2$ such that
$$
\lim_{x\to \infty}\frac{max\{ A(x),B(x)\}\sqrt{\log \min \{  A(x),B(x)\} } }{x}=\frac{1}{\sqrt{\pi }} .
$$
\end{kov}
Inspired by this, Theorem \ref{limsup:maxlogmin} may be sharpened as follows:
\begin{prob}
Is it true that for any infinite $(A,B)\in MC_2$
$$
\limsup_{x\to \infty}\frac{max\{ A(x),B(x)\}\sqrt{\log \min \{  A(x),B(x)\} } }{x}>0?
$$
\end{prob}
It follows from Theorem \ref{vegeseset} that
$$
\frac{c_1}{\log k}\le \inf_{B:|B|=k}\{ \inf_{A:(A,B)\in MC_2}\{ \limsup_{x\to \infty} \frac{A(x)}{x} \} \} \le \frac{c_2}{\log \log k}.
$$
It would be interesting to determine the right magnitude of the above infimum.
\begin{prob}
  Find the magnitude of the function
$$
f(k)=\inf_{B:|B|=k}\{ \inf_{A:(A,B)\in MC_2}\{ \limsup _{x\to \infty }\frac{A(x)}{x}  \}  \}.
$$
\end{prob}
In Theorem \ref{limsupfxyes} we have some smoothness assumptions. Is it true that we can omit these conditions?
\begin{prob}
  Is it true that if $f(x)>0$, $x\ge 1$ is a monotonically increasing function such that $\sum_{n=1}^{\infty }\frac{f(n)}{n^2}=\infty $, then there exists an $(A,B)\in MC_2$ such that $A(x)=O(f(x))$ and $B(x)=o(x)$?
\end{prob}

\section{Proofs}

Throughout the paper the set of primes is denoted by $P$ and $p$ will always denote a prime number.

\begin{proof}[Proof of Theorem \ref{liminf:maxlogmin}]
The set $S\subseteq \mathbb{Z}^+$ is called multiplicative basis of order 2 if every positive integer $n$ can be written as $n=ss'$, where $s,s'\in S$. The set of multiplicative bases of order 2 is denoted by $MB_2$. Pach and S\' andor \cite{PS} proved that if $S\in MB_2$, then $\displaystyle \liminf_{x\to \infty} \frac{S(x)}{\frac{x}{\log x}}>1$. For $(A,B)\in MC_2$, we have $A\cup B\in MB_2$ and $(A\cup B)(x)\le A(x)+B(x)$. It follows that
\[\liminf\limits_{x\to \infty} \frac{A(x)+B(x)}{\frac{x}{\log x}}>1,\]
that is, there exists a $\delta >0$ such that
\begin{align}\label{osszeg}
\liminf\limits_{x\to \infty} \frac{A(x)+B(x)}{\frac{x}{\log x}}\ge 1+\delta .
\end{align}
Let $\beta >0$ be a constant such that  $$\frac{1}{2+\delta }<\frac{1}{2}-\beta<\frac{1}{2+\frac{\delta }{2} }.$$

In order to prove the theorem, we define a partition of $\mathbb{Z}^+$. Let
\[H_1=\{x\in \mathbb{Z}^+: \frac{\delta}{2} \frac{x}{\log x}\le \min\{A(x),B(x)\} \}\]
\[H_2=\{x\in \mathbb{Z}^+: \frac{x^{\frac{1}{2}-\beta}}{\log^2 x}\le \min\{A(x),B(x)\} <\frac{\delta}{2} \frac{x}{\log x}\} \]
\[H_3=\{x\in \mathbb{Z}^+:\min\{A(x),B(x)\}< \frac{x^{\frac{1}{2}-\beta}}{\log^2 x}\} .\]
Because of (\ref{osszeg}),
\[\liminf_{x\in H_1, x\to \infty} \frac{ \max\{A(x),B(x)\}}{\frac{x}{\log x}}\ge \frac{1}{2}+\frac{\delta}{2}\]
It is easy to see from the definition of $H_1$ that
\[\liminf_{x\in H_1, x\to \infty} \frac{\log  \min\{A(x),B(x)\}}{\log x}= 1.\]
Hence
\[\liminf_{x\in H_1, x\to \infty} \frac{ \max\{A(x),B(x)\}\log  \min\{A(x),B(x)\}}{x}\ge \frac{1}{2}+\frac{\delta}{2}.\]
In the second case, it follows from (\ref{osszeg}) that
\[\liminf_{x\in H_2, x\to \infty} \frac{ \max\{A(x),B(x)\}}{\frac{x}{\log x}}\ge 1+\frac{\delta}{2}\]
and it is easy to see from the definition of $H_2$ that
\[\liminf_{x\in H_2, x\to \infty} \frac{\log  \min\{A(x),B(x)\}}{\log x}\ge \frac{1}{2}-\beta .\]
Hence
\[\liminf_{x\in H_2, x\to \infty} \frac{ \max\{A(x),B(x)\}\log  \min\{A(x),B(x)\}}{x}\ge \left(1+\frac{\delta}{2}\right)\left(\frac{1}{2}-\beta \right)>\frac{1}{2}.\]
It remains to consider the case $x\in H_3$ as $x\to \infty$. In this case
 $$ \min\{A(x),B(x)\}\log^2 \min\{A(x),B(x)\}<x^{\frac{1}{2}-\beta}.$$
Let us suppose that $\min\{A(x),B(x)\}=B(x)$ for a given $x\in H_3$. Then
\begin{equation}\label{lower}
\frac{\log x}{\log (B(x)\log^2 B(x))}\ge 2+\frac{\delta }{2}.
\end{equation}
Clearly,
	\begin{equation}\label{upperx}
\begin{split}
		x &\le\phantom{=} |\{n\le x: \exists (a,b)\in A\times B\colon n = ab,b\ge  B(x)\log^2 B(x)\}|+ \\
		&\phantom{=}+ |\{n\le x: \exists (a,b)\in A\times B\colon n = ab,1<b<  B(x)\log^2 B(x)\}|+ \\
		&\phantom{=}+ |\{n\le x: \exists (a,b)\in A\times B\colon n = ab,b=1\}|.
\end{split}
	\end{equation}
	The first term in (\ref{upperx}) can be estimated by
	\begin{align*}
		|\{n\le x: \exists (a,b)\in A\times B\colon n = ab,b\ge B(x)\log^2 B(x)\}|&\le \sum_{B(x)\log^2B(x)\le b\le x, b\in B}\frac{x}{b}\\&\le \frac{x}{B(x)\log^2B(x)}B(x)
		\le  \frac{x}{\log^2B(x)}.
	\end{align*}
	The second term in (\ref{upperx}) can be estimated by
	\begin{align*}
		|\{n\le x:&\exists (a,b)\in A\times B\colon n = ab,1<b< B(x)\log^2 B(x)\}|\le
		\\
		&\le |\{n\le x: \exists p\in P\colon p\mid n ,1<p< B(x)\log^2 B(x)\}|.
	\end{align*}
The third term in (\ref{upperx}) is
\[|\{n\le x\mid\exists (a,b)\in A\times B\colon n = ab,b=1\}|=A(x),\]
and so
\begin{align}\label{upperx3}
    x\le \frac{x}{\log^2 B(x)}+|\{n\le x: \exists p\in P\colon p\mid n ,1<p< B(x)\log^2 B(x)\}|+A(x)
\end{align}
On the other hand,
\begin{equation}\label{lowerAx}
\begin{split}
    x=&1+|\{n\le x: \exists p\in P\colon p\mid n ,1<p< B(x)\log^2 B(x)\}|\\ &+|\{n\le x\mid \forall p\in P\colon p\mid n\implies p\ge B(x)\log^2B(x)\}|.
\end{split}
\end{equation}
It follows from (\ref{upperx3}) and (\ref{lowerAx}) that
\begin{equation}\label{lowerA}
		|\{n\le x: \forall p\in P\colon p\mid n\implies p\ge B(x)\log^2B(x)\}|+1-\frac{x}{\log^2B(x)}\le A(x).
\end{equation}
To estimate the first term we need the  Buchstab function. The Buchstab function is the unique continuous function $\omega : [1,\infty [ \to \mathbb{R}^+$ defined by the delay differential equation
$$
\omega (u)=\frac{1}{u} \quad \hbox{for $1\le u\le 2$}
$$   	
$$
\frac{d}{du}(u\omega (u))=\omega (u-1)\quad \hbox{for $u\ge 2$.}
$$
In the second equation, the derivative at $u=2$ should be taken as $u$ approaches $2$ from the right. Denote by $\varphi (x,y)$  the number of positive integers less than or equal to $x$ with no prime factor less than $y$. 	
In 1970, Warlimont \cite{W} proved the following remarkable result.
\begin{tet}[Warlimont, 1970]\label{War}
Let $u_0>1$. There exists a $C_0=C_0(u_0)$ such that if $y\ge 2$, $\frac{\log x}{\log y}\ge u_0$ then
$$
|\varphi (x,y)-e^{\gamma}\omega (u)x\prod_{p<y}(1-\frac{1}{p})|\le \frac{C_0x\prod_{p<y}\left( 1-\frac{1}{p}\right) }{\log x}.
$$
\end{tet} 	
It is well known that the minimum of $\omega (u)$ is taken at $u=2$, $\omega (2)=0.5$ and for every $\mu >0$, there exists an $\varepsilon >0$ such that for $u\ge 2+\mu $, we have $\omega (u)\ge 0.5+\varepsilon $. Let us suppose that $\omega (u)\ge 0.5+\varepsilon _0$ if $u\ge 2+\frac{\delta }{2}$. According to the Warlimont's theorem and (\ref{lower}) we get that
$$
|\{n\le x\mid \forall p\in P\colon p\mid n\implies p\ge B(x)\log^2B(x)\}|=\varphi(x,B(x)\log^2B(x))\ge
$$
$$
e^{\gamma }\omega \left( \frac{\log x}{\log (B(x)\log ^2B(x))} \right) x \prod_{p<B(x)\log^2B(x)}(1-\frac{1}{p}) -\frac{C_0x\prod_{p<B(x)\log ^2B(x)}\left( 1-\frac{1}{p}\right) }{\log x}\ge
$$
$$
(0.5+\varepsilon _0+o(1))e^{\gamma }x\prod_{p<B(x)\log^2B(x)}(1-\frac{1}{p}),
$$
where $\gamma$ is the Euler-Mascheroni constant.
	
	
	
	It is well-known that \[\prod_{p<y,\: p\in P} (1-\frac{1}{p})=(1+o(1))\frac{e^{-\gamma}}{\log y}\]
	as $y\to \infty$, so
		\[|\{n\le x\mid \forall p\in P\colon p\mid n\implies p\ge B(x)\log^2B(x)\}|\ge (0.5+\varepsilon _0+o(1))\frac{x}{\log\left( B(x)\log^2B(x)\right) }=$$ $$(0.5+\varepsilon _0+o(1))\frac{x}{\log B(x)} \]
	as $x\to \infty $. By (\ref{lowerA}),
$$
A(x)\ge 1-\frac{x}{\log^2B(x)}+(0.5+\varepsilon _0+o(1))\frac{x}{\log B(x)}=(0.5+\varepsilon _0+o(1))\frac{x}{\log B(x)}.
$$	
Hence we have
$$
(0.5+\varepsilon _0+o(1))x\le A(x)\log B(x).
$$
It follows that
\[\liminf_{x\in H_3, x\to \infty} \frac{ \max\{A(x),B(x)\}\log  \min\{A(x),B(x)\}}{x}\ge 0.5+\varepsilon _0.\]

	
	
	

	
	
	
	
	
	
	
	
	This means that
	
	\[\liminf\limits_{x\to \infty} \frac{\max\{A(x),B(x)\}\log \min\{A(x),B(x)\}}{x}\ge \min\{\frac{1}{2}+\frac{\delta }{2} , (1+\frac{\delta}{2})(\frac{1}{2}-\beta), \frac{1}{2}+\varepsilon _0\}>\frac{1}{2},\]
which completes the proof.	
	
	
	
	


\end{proof}

\begin{proof}[Proof of Theorem \ref{limsup:maxlogmin}]

If $A(x)\neq o(x)$ or $B(x)\neq o(x)$, then $\max\{ A(x),B(x)\} \ne o(x)$ and $\min \{ A(x),B(x)\} \to \infty$, the conclusion directly follows. So we may suppose that $A(x)=o(x)$ and $B(x)=o(x)$. It follows from Theorem~\ref{limsup:max} that
\begin{equation}\label{supmax}
\limsup_{x\to \infty}\frac{\max \{A(x),B(x)\}}{\frac{x}{\sqrt{\log x}}}\geq \frac{1}{\sqrt{\pi}}.
\end{equation}
If $\displaystyle \liminf_{x\to \infty}\frac{\log \min \{A(x),B(x)\} }{\sqrt{\log x}}=\infty$, then we are done. So we can suppose that there exists a constant $D$ such that for infinitely many $x\in \mathbb{Z}^+$,
$$\frac{\textrm{min}\{A(x),B(x)\}}{ e^{D\sqrt{\log x}}}<1.$$

Without loss of generality, we can assume that
$$\frac{B(x)}{ e^{D\sqrt{\log x}}}<1$$
for infinitely many $x\in \mathbb{Z}^+$.

If $A(x)=o(x)$, then there exist infinitely many $x\in \mathbb{Z}^+$ such that $B(x)>\sqrt{x}$, otherwise for $B=\{ b_1,b_2,\dots \}$, $b_1<b_2<\dots $ there exists an $n_0$ such that for $n\ge n_0$, we have $n=B(b_n)\le \sqrt{b_n}$, that is $b_n\ge n^2$. It follows that $$\sum_{b\in B}\frac{1}{b}=\sum_{n=1}^{n_0-1}\frac{1}{b_n}+\sum_{n=n_0}^{\infty}\frac{1}{b_n}\le \sum_{n=1}^{n_0-1}\frac{1}{b_n}+\sum_{n=n_0}^{\infty}\frac{1}{n^2}< \infty ,$$
therefore $\sum_{b\in B}\frac{1}{b}$ convergent. Then there exists an $N_1$ such that $\sum_{b\ge N_1,b\in B}\frac{1}{b}<\frac{1}{2}$. Clearly
$$
x\le |\{ n:n\le x, \exists b<N_1, b\in B, a\in A, n=ab\} |+|\{ n:n\le x, \exists b\ge N_1, b\in B, a\in A, n=ab  |\le
$$
$$
\sum_{b<N_1,b\in B}A(\frac{x}{b})+\sum_{b\ge N_1,b\in B}A(\frac{x}{b})\le o(x)+\sum_{b\ge N_1,b\in B}\frac{x}{b}\le (0.5+o(1))x,
$$
a contradiction.




So we may assume that there are infinitely many positive integers $y_1<y_2<y_3< \ldots$ such that
$0.5 e^{D\sqrt{\log y_n}}<B(y_n)< e^{D\sqrt{\log y_n}}$ and $x_1<x_2<x_3<\dots $ such that $0.5\sqrt{x_n}<B(x_n)<\sqrt{x_n}$ with $e^{\sqrt{\log x_n}}>y_n$ for every $n$.

We are going to estimate $A(x_n)$. We need the following sets:

$C_n=\{p:p>y_n \textrm{ and }\exists m\in B\cap \{1,\dots ,x_n\} \textrm{ such that }  p \textrm{ is the largest prime factor of } m\},$

$D_n=\{p:p<y_n \textrm{ and } p\in B \},$

$E_n=\{p:p<y_n \}.$

It is easy to check that if $k$ is a positive integer such that $k<x_n$ and $k=pq$, where $p$ and $q$ are primes, $p\in E_n\setminus D_n$, $q\notin C_n$ and $q>y_n$, then $k$ must belong to the set $A$. Now we want to estimate the number of such $k$'s. Clearly,
$$A(x_n)\geq \sum\limits_{p\in E_n\setminus D_n} \left(\pi\left(\frac{x_n}{p}\right)-|C_n|-\pi(y_n)\right) $$
If $x_n$ is large enough, then for any $p\in E_n\setminus D_n$,  \[\pi\left(\frac{x_n}{p}\right)=(1+o(1))\frac{x_n/p}{\log (x_n/p)}=(1+o(1))\frac{1}{p}\frac{x_n}{\log x_n -\log p}>(1+o(1))\frac{1}{p}\frac{x_n}{\log x_n}\ge
\frac{1}{p}\frac{x_n}{2\log x_n}.\]
Clearly,
\[|C_n|\le B(x_n)<\sqrt{x_n},\]
and
\[\pi(y_n)\le y_n<e^{\sqrt{\log x_n}}.\]
It follows that
\[A(x_n) > \sum\limits_{p\in E_n\setminus D_n} \left( \frac{1}{p}\frac{x_n}{2\log x_n}-\sqrt{x_n}-e^{\sqrt{\log x_n}}\right) > \frac{x_n}{3\log x_n}\sum\limits_{p\in E_n\setminus D_n} \frac{1}{p}\]
if $n$ is large enough. It is well known that $\sum\limits_{p<N}\frac{1}{p}=\log\log N+ O(1)$ and $\sum\limits_{i=1}^{N}\frac{1}{p_i}=\log\log N+O(1)$ as well. Since $|D_n|\le B(y_n)<e^{D\sqrt{\log y_n}}$, there are at most $e^{\sqrt{\log y_n}}$ primes in the set $D_n$, so $\sum_{p\in D_n}\frac{1}{p}\le \sum\limits_{i=1}^{e^{\sqrt{\log y_n}}} \frac{1}{p_i}$. Therefore
$$\sum\limits_{p\in E_n\setminus D_n} \frac{1}{p}\geq \sum\limits_{p\in E_n} \frac{1}{p} -\sum\limits_{i=1}^{e^{\sqrt{\log y_n}}} \frac{1}{p_i}=\log \log y_n-\log \log e^{\sqrt{\log y_n}}+ O(1)= \frac{\log \log y_n}{2}+O(1).$$
It follows that $$A(x_n)\geq \frac{x_n}{6\log x_n}(\log\log y_n+O(1)).$$ We have supposed that $B(x_n)\ge 0.5\sqrt{x_n}$. Hence
$$\frac{A(x_n)\log B(x_n)}{x_n}\geq \frac{\log\log y_n}{12}+O(1).$$
This completes the proof, because $y_n\to \infty$ as $n\to \infty $.
\end{proof}

It order to prove Theorems \ref{liminffx} and \ref{limsupfxyes} we need the following lemma.

\begin{lem}\label{Q}
  Let $Q$ be a subset of prime numbers such that $\sum_{q\in Q}\frac{1}{q}=\infty $. Let
  $$
  A=\{ n: \hbox{ $n$ is squarefree }, n\in \mathbb{Z}^+, p|n \Rightarrow p\in Q\}.
  $$
  Then there exists a set $B\subseteq \mathbb{Z}^+$ such that $(A,B)\in MC_2$ and $B(x)=o(x)$.
\end{lem}
\begin{proof}
  Let $Q=\{ q_1,q_2,\dots  \}$, $q_1<q_2<\dots $. Define $Q_k=\{ q_1,\dots ,q_k \} $ for every $k\in \mathbb{Z}^+$.  Let
  $$
  A_k=\{ n: \hbox{ $n$ is squarefree }, n\in \mathbb{Z}^+, p|n \Rightarrow p\in Q_k \}.
  $$
  Let
  $$
  B_k=\{n: n=ml^2, \hbox{$m$ is squarefree and }p|m \Rightarrow p\notin Q_k \} .
  $$
  Then $(A_k,B_k)\in MC_2$, $A_1\subseteq A_2\subseteq \dots $, $\displaystyle A=\cup _{k=1}^{\infty }A_k$ and $B_1\supseteq B_2\supseteq \dots $. Let $\displaystyle B=\cap _{k=1}^{\infty}B_k .$ Then $(A,B)\in MC_2$. We are going to prove that for any $\varepsilon >0$, there exists an $x_0=x_0(\varepsilon )$ such that
  \begin{equation}\label{upperB}
     B(x)\le \varepsilon x
  \end{equation}
  for $x\ge x_0$.

  Let $k$ be an integer such that $4e^{-\sum_{i=1}^k\frac{1}{2q_i}}<\varepsilon $. Let $1=b_1^{(k)}<b_2^{(k)}<\dots $ be the squarefree integers in the set $B_k$. Then
  $$
  B(x)\le B_k(x)\le \sum_{b_l^{(k)}\le x}\sqrt{\frac{x}{b_l^{(k)}}}.
  $$
It follows from $\gcd (b_l^{(k)},q_1q_2\dots q_k)=1$  $$|\{ b_1^{(k)},b_2^{(k)},\dots  \}  \cap \{ Nq_1\dots q_k+1,\dots ,(N+1)q_1\dots q_k\} |\le (q_1-1)\dots (q_k-1)$$ for any nonnegative integer $N$. If $M(q_1-1)\dots (q_k-1)\le l <(M+1)(q_1-1)\dots (q_k-1)$, then
  $$
  b_l^{(k)}\ge (M+1)q_1\dots q_k-q_1\dots q_k\ge l\prod_{i=1}^k\frac{q_i}{q_i-1}-q_1\dots q_k\ge le^{\sum_{i=1}^k\frac{1}{q_i}}-q_1\dots q_k.
  $$
  If $l>2q_1\dots q_k$, then
  $$
  \frac{1}{\sqrt{b_l^{(k)}}}\le \frac{1}{\sqrt{le^{\sum_{i=1}^k\frac{1}{q_i}}-q_1\dots q_k}}= \frac{1}{\sqrt{le^{\sum_{i=1}^k\frac{1}{q_i}} (1-\frac{q_1\dots q_k}{le^{\sum_{i=1}^k\frac{1}{q_i}}})}}\le
  $$
  $$
  \frac{1}{\sqrt{l}e^{\sum_{i=1}^k\frac{1}{2q_i}} \left( 1-\frac{q_1\dots q_k}{le^{\sum_{i=1}^k\frac{1}{q_i}}}\right) } \le \frac{1}{\sqrt{l}e^{\sum_{i=1}^k\frac{1}{2q_i}}}+\frac{1}{\sqrt{l}e^{\sum_{i=1}^k\frac{1}{2q_i}}}\frac{2q_1\dots q_k}{le^{\sum_{i=1}^k\frac{1}{q_i}}}.
  $$
  Hence
  $$
  B(x)\le \sum_{l\le 2q_1\dots q_k}\sqrt{\frac{x}{b_l^{(k)}}}+\sum_{2q_1\dots q_k<l\le x}\sqrt{\frac{x}{b_l^{(k)}}}\le O(\sqrt{x})+\sum_{2q_1\dots q_k<l\le x } \left ( \frac{\sqrt{x}}{\sqrt{l}e^{\sum_{i=1}^k\frac{1}{2q_i}}}+\frac{\sqrt{x}}{l^{3/2}e^{\sum_{i=1}^k\frac{1}{2q_i}}} \right)\le
  $$
  $$
  O(\sqrt{x})+\frac{\sqrt{x}}{e^{\sum_{i=1}^k\frac{1}{2q_i}}}\int_0^x\frac{1}{\sqrt{t}}dt+O(1)=O(\sqrt{x})+2xe^{-\sum_{i=1}^k\frac{1}{2q_i}}\le \frac{\varepsilon }{2}x+O(\sqrt{x}),
  $$
  therefore there exists an $x_0=x_0(\varepsilon )$ such that (\ref{upperB}) holds.
\end{proof}

\begin{proof}[Proof of Theorem \ref{liminffx}]
Let us choose sequences $M_1<M_2<\dots $ and $N_1<N_2<\dots $ such that $M_k<N_k$,  $\displaystyle \sum_{M_k< p\le N_k}\frac{1}{p}\ge 1$, $\displaystyle \prod_{p\le N_{k-1}}p<M_k$ and $k2^{N_{k-1}}<f(M_k)$  for every $k\in \mathbb{Z}^+$. Let
$$
Q=\{p: \hbox{ there exists a $k\in  \mathbb{Z}^+$ such that } M_k<p\le N_k\}.
$$
Let
$$
A=\{n: n\in \mathbb{Z}^+,\hbox{ $n$ is squarefree } p|n \Rightarrow p\in Q \} .
$$
By construction, we have $\displaystyle \sum_{q\in Q}\frac{1}{q}=\sum_{k=1}^{\infty}\sum_{M_k<p\le N_k}\frac{1}{p}\ge \sum_{k=1}^{\infty}1=\infty $. By Lemma \ref{Q}, we know that there exists a set $B\in \mathbb{Z}^+$ such that $(A,B)\in MC_2$ and $B(x)=o(x)$ as $x\to \infty$. Moreover,
$$
\frac{A(M_k)}{f(M_k)}\le \frac{2^{\pi (N_{k-1})}}{f(M_k)} < \frac{2^{N_{k-1}}}{k2^{N_{k-1}}}=\frac{1}{k}
$$
for every $k\in \mathbb{Z}^+$, therefore
$$
\liminf_{x\to \infty }\frac{A(x)}{f(x)}=0,
$$
which completes the proof.
\end{proof}

\begin{proof}[Proof of Theorem \ref{limsupfxno}]
Let us suppose that for the function $f(x)>0$, $x\ge 1$ the series $\displaystyle \sum_{n=1}^{\infty}\frac{f(n)}{n^2}$ is finite. Let us suppose that for some $(A,B)\in MC_2$, we have
$A(x)\le Cf(x)$ for every $x\in \mathbb{Z}^+$ and $B(x)=o(x)$. Then
$$
\sum_{a\le N, a\in A}\frac{1}{a}=\sum_{n=1}^N\frac{A(n)-A(n-1)}{n}=\sum_{n=1}^{N-1}\frac{A(n)}{n(n+1)}+\frac{A(N)}{N}\le 1+C\sum_{n=1}^{\infty}\frac{f(n)}{n^2},
$$
therefore the series $\sum_{a\in A}\frac{1}{a}$ converges. It follows that there exists an $n_1$ such that $\sum_{a\ge n_1,a\in A}\frac{1}{a}<\frac{1}{2}$. Clearly
$$
x\le |\{ n:n\le x, \exists a<n_1, a\in A, b\in B, n=ab\} |+|\{ n:n\le x, \exists a\ge n_1, a\in A, b\in B, n=ab  |\le
$$
$$
\sum_{a<n_1,a\in A}B(\frac{x}{a})+\sum_{a\ge n_1,a\in A}B(\frac{x}{a})\le o(x)+\sum_{a\ge n_1,a\in A}\frac{x}{a}\le (0.5+o(1))x,
$$
a contradiction.
\end{proof}

\begin{proof}[Proof of Theorem \ref{limsupfxyes}]
We are going to find a set $Q\subseteq P$ such that
\begin{equation}\label{Qinfty}
\sum_{q\in Q}\frac{1}{q}=\infty
\end{equation}
and the set
$$
A=\{ n: n\hbox{ is squarefree }p|n \Rightarrow p\in Q  \}
$$
satisfies
\begin{equation}\label{AFx}
  A(x)=O(f(x)).
\end{equation}
Then by Lemma \ref{Q}, there exists a $B\subseteq \mathbb{Z}^+$ such that $(A,B)\in MC_2$ and $B(x)=o(x)$.

Let $\chi _Q(n)= \left\{ \begin{array}{ll} 1, & \mbox{if $n\in Q$}\\ 0, & \mbox{ if $n\notin Q$} \end{array} \right. $. Let
$$
A_k=\{ n: n\in A, n\hbox{ has $k$ different prime factors}\}.
$$
To prove (\ref{AFx}) it is enough to verify
\begin{equation}\label{Akfx}
A_k(x)\le \frac{e^{\frac{1}{c_1}}f(x)\left( \frac{2}{c_1}\sum_{n<x}\frac{\chi _Q(n)}{n} \right) ^{k-1} }{(k-1)!exp\left( \frac{2}{c_1}\sum_{n<x}\frac{\chi _Q(n)}{n} \right) }
\end{equation}
holds for $x\ge 1$ and $k\in \mathbb{Z}^+$, because
$$
A(x)\le 1+\sum_{k=1}^{\infty }A_k(x)\le 1+\sum_{k=1}^{\infty }\frac{e^{\frac{1}{c_1}}f(x)\left( \frac{2}{c_1}\sum_{n<x}\frac{\chi _Q(n)}{n} \right) ^{k-1} }{(k-1)!exp\left( \frac{2}{c_1}\sum_{n<x}\frac{\chi _Q(n)}{n} \right) }=1+e^{\frac{1}{c_1}}f(x).
$$
First, we define a set $R\subseteq P$, then by truncating it we get the new set $Q$. The set $R$ is defined recursively. We run over the primes $p$ and the prime number $p$ is included in the set $R$ if and only if
$$
R(p-1)+1\le \frac{f(p)}{exp\left( \frac{2}{c_1}\sum_{n<p}\frac{\chi _R(n)}{n} \right)}.
$$
First, we show that
$$
\sum_{r\in R}\frac{1}{r}=\infty .
$$
By contradiction, let us suppose that $\sum_{r\in R}\frac{1}{r}=K<\infty $. Let us suppose that $p\in R$ and the next prime in the set $R$ is $p'$. Let $p\le x<p'$. Then
$$
R(x)=R(p)=R(p-1)+1\le \frac{f(p)}{exp\left( \frac{2}{c_1}\sum_{n<p}\frac{\chi _R(n)}{n} \right)}\le \frac{f(x)}{exp\left( \frac{2}{c_1}\sum_{n<x}\frac{\chi _R(n)}{n} -\frac{2}{c_1}\frac{1}{p} \right)} =
$$
$$
\frac{e^{\frac{2}{c_1}\frac{1}{p}}f(x)}{exp\left( \frac{2}{c_1}\sum_{n<x}\frac{\chi _R(n)}{n} \right)}\le
\frac{e^{\frac{1}{c_1}}f(x)}{exp\left( \frac{2}{c_1}\sum_{n<x}\frac{\chi _R(n)}{n} \right)}\le e^{\frac{1}{c_1}}f(x).
$$
It follows that for any $x\ge 1$,
$$
R(x)\le e^{\frac{1}{c_1}}f(x)=o(\frac{x}{\log x}).
$$
It is well known that $\pi (n)-\pi (\frac{n}{2})=(0.5+o(1))\frac{n}{\log n}$ as $n\to \infty$, and so there exists an $n_1$ such that for any $n\ge n_1$ there exists a prime number $p$ in the interval $[0.5n,n]$ such that $p\notin R$. Then by the definition of set $R$ and condition 3,
$$
R(p-1)+1>\frac{f(p)}{exp\left( \frac{2}{c_1}\sum_{n<p}\frac{\chi _R(n)}{n} \right)}> \frac{f(p)}{e^{\frac{2K}{c_1}}}= \frac{\frac{f(p)}{f(n)}}{\frac{p}{n}}\frac{p}{n}\frac{1}{e^{\frac{2K}{c_1}}}f(n)\ge \frac{c_1f(n)}{2e^{\frac{2K}{c_1}}}.
$$
If $f(n)>\frac{6e^{\frac{2K}{c_1}}}{c_1}$, then
\begin{equation}\label{upperR}
R(n)\ge R(p-1)\ge \frac{c_1f(n)}{2e^{\frac{2K}{c_1}}}-1\ge \frac{c_1f(n)}{2e^{\frac{2K}{c_1}}}-\frac{c_1f(n)}{6e^{\frac{2K}{c_1}}}  \ge \frac{f(n)}{3c_2e^{\frac{2K}{c_1}}}.
\end{equation}
So inequality (\ref{upperR}) holds if $n\ge n_2$. By condition 4, we get that $\sum_{n=1}^{\infty }\frac{R(n)}{n^2}=\infty $, but
$$
\sum_{r\le N, r\in R}\frac{1}{r}=\sum_{n=1}^N\frac{R(n)-R(n-1)}{n}=\sum_{n=1}^{N-1}\frac{R(n)}{n(n+1)}+\frac{R(N)}{N}\ge \frac{1}{2}\sum_{n=1}^{N-1}\frac{R(n)}{n^2}\to \infty
$$
as $N\to \infty $, a contradiction.

Let $R=\{ r_1,r_2,\dots \}$, $r_1<r_2<\dots $. It is well known that $\displaystyle \lim _{x\to \infty }\sum_{\sqrt{x}\le p\le x}\frac{1}{p}=\log 2$. We can find an $N_0\in \mathbb{Z}^+$ such that
\begin{equation}\label{rare}
\sum_{\sqrt{x}\le r_i\le x, N_0|i}\frac{1}{r_i}<\frac{c_1\log 2}{2}
\end{equation}
for every $x\ge 1$. Let
$$
Q=\{ r_{N_0},r_{2N_0},\dots \}.
$$
Clearly
$$
\sum_{q<x, q\in Q}\frac{1}{q}\ge \frac{\sum_{r<x,r\in R}\frac{1}{r}-\sum_{i=1}^{N_0-1}\frac{1}{r_i}}{N_0}\to \infty
$$
as $N\to \infty$, therefore (\ref{Qinfty}) holds.

It remains to prove (\ref{Akfx}). We use by induction on $k$.

For $k=1$,
$$
A_1(x)=Q(x)\le R(x)\le \frac{e^{\frac{1}{c_1}}f(x)}{exp\left( \frac{2}{c_1}\sum_{n<x}\frac{\chi _R(n)}{n} \right)}\le \frac{e^{\frac{1}{c_1}}f(x)}{exp\left( \frac{2}{c_1}\sum_{n<x}\frac{\chi _Q(n)}{n} \right)}.
$$
Let us suppose that the statement holds for $k$ and we prove it for $k+1$. If $n\le x$, $n \in A_{k+1}$, that is $n=p_1\dots p_{k+1}$, $p_1<\dots <p_{k+1}$, $p_i\in Q$, then $p_i<\sqrt{x}$ for $1\le i\le k$, therefore $kA_{k+1}(x)\le \sum_{q\in Q, q<\sqrt{x} }A_k(\frac{x}{q})$. Hence
$$
kA_{k+1}(x)\le \sum_{q\in Q, q<\sqrt{x} }A_k(\frac{x}{q})\le \sum_{q\in Q, q<\sqrt{x} }\frac{e^{\frac{1}{c_1}}f(\frac{x}{q})\left( \frac{2}{c_1}\sum_{n<\frac{x}{q}}\frac{\chi _Q(n)}{n} \right) ^{k-1} }{(k-1)!exp\left( \frac{2}{c_1}\sum_{n<\frac{x}{q}}\frac{\chi _Q(n)}{n} \right) }.
$$
By condition 3, we have
$$
f(\frac{x}{q})\le \frac{1}{c_1}\frac{1}{q}f(x)
$$
and
by (\ref{rare}),
$$
exp\left( \frac{2}{c_1}\sum_{n<\frac{x}{q}}\frac{\chi _Q(n)}{n} \right)=\frac{exp\left( \frac{2}{c_1}\sum_{n<x}\frac{\chi _Q(n)}{n} \right)}{exp\left( \frac{2}{c_1}\sum_{\frac{x}{q}\le n<x}\frac{\chi _Q(n)}{n} \right)}\ge \frac{exp\left( \frac{2}{c_1}\sum_{n<x}\frac{\chi _Q(n)}{n} \right)}{exp\left( \frac{2}{c_1}\sum_{\sqrt{x}\le n<x}\frac{\chi _Q(n)}{n} \right)} \ge $$
$$
 \frac{exp\left( \frac{2}{c_1}\sum_{n<x}\frac{\chi _Q(n)}{n} \right)}{exp\left( \frac{2}{c_1}\frac{c_1\log 2}{2} \right)} =\frac{exp\left( \frac{2}{c_1}\sum_{n<x}\frac{\chi _Q(n)}{n} \right)}{2}.
$$
Hence
$$
kQ_{k+1}(x)\le \sum_{q\in Q, q<\sqrt{x} }\frac{e^{\frac{1}{c_1}}\frac{1}{c_1}\frac{1}{q}f(x)\left( \frac{2}{c_1}\sum_{n<\frac{x}{q}}\frac{\chi _Q(n)}{n} \right) ^{k-1} }{(k-1)!\frac{1}{2}exp\left( \frac{2}{c_1}\sum_{n<x}\frac{\chi _Q(n)}{n} \right) }\le
$$
$$
\frac{e^{\frac{1}{c_1}}f(x)\left( \frac{2}{c_1}\sum_{n<x}\frac{\chi _Q(n)}{n} \right) ^{k-1} }{(k-1)!exp\left( \frac{2}{c_1}\sum_{n<x}\frac{\chi _Q(n)}{n} \right) }\sum_{q\in Q, q<\sqrt{x} }\frac{2}{c_1}\frac{1}{q}<\frac{e^{\frac{1}{c_1}}f(x)\left( \frac{2}{c_1}\sum_{n<x}\frac{\chi _Q(n)}{n} \right) ^k }{(k-1)!exp\left( \frac{2}{c_1}\sum_{n<x}\frac{\chi _Q(n)}{n} \right) },
$$
which completes the proof.
\end{proof}

\begin{proof}[Proof of Theorem \ref{vegeseset}]

First we prove the first part of Theorem \ref{vegeseset}. Let \[V=\{p\: \text{prime }: \exists n\in B,  p\hbox{ is the largest prime factors of }n\}.\]

If no prime in $V$ divides a number $n$, then equality $n=ab, a\in A, B\in B$ implies $b=1, a=n$, thus $n\in A$. Hence
\[\liminf_{x\to \infty} \frac{A(x)}{x}\ge \prod_{p\in V}\left(1-\frac{1}{p}\right)\ge \prod_{i=1}^{|B|}\left(1-\frac{1}{p_i}\right), \]
where $p_i$ is the $i$th smallest prime number.

It is well known that
\begin{equation}\label{mertens}
\lim_{k\to \infty} \log k\prod_{i=1}^k \left(1-\frac{1}{p_i}\right)=e^{-\gamma},
\end{equation}
where $\gamma$ is the Euler-Mascheroni constant. It follows that there exists a constant $c_1>0$ such that
\[\liminf_{x\to \infty} \frac{A(x)\log |B|}{x}\ge c_1.\]

Now we prove the second part of Theorem \ref{vegeseset}. Let $2^k$ be the largest power of $2$ that is less than or equal to $N$. Let $B\subseteq \mathbb{Z}^+$ be a set that contains the set of all square-free numbers composed of the $k$ smallest primes. Clearly $2^{k-1}<|B|=2^k\le N$. Let
\[A=\{mr^2: p_i\nmid m \:\forall \: 1\le i\le k, r\in \mathbb{Z}^+ \}.\]
Denote by $A_{sf}$ the set of squarefree integers in the set $A$. If $n\in A_{sf}$, then $\gcd (n,p_1\dots p_k)=1$, therefore $A_{sf}(x)\le O(1)+x\prod_{i=1}^k\left( 1-\frac{1}{p_i} \right) $. Hence
$$
A(x)\le \sum_{r\le \sqrt{x}}A_{sf}\left( \frac{x}{r^2} \right)\le \sum_{r\le \sqrt{x}}\left( \frac{x}{r^2}\prod_{i=1}^k\left( 1-\frac{1}{p_i}\right) +O(1)  \right) = O(\sqrt{x})+\frac{\pi ^2 x}{6}\prod_{i=1}^k\left( 1-\frac{1}{p_i} \right) .
$$
Then $A$ and $B$ are multiplicative complements and by (\ref{mertens}),
\[\limsup_{x\to \infty} \frac{A(x)}{x}=\frac{\pi ^2}{6}\prod_{i=1}^k \left(1-\frac{1}{p_i}\right)\le \frac{c'}{\log k}\le \frac{c_2}{\log \log N},\]
which completes the proof.

\end{proof}

\renewcommand{\refname}{Bibliography}

\end{document}